\documentclass[a4paper,12pt]{amsart}
\usepackage{amssymb, amsmath, amsthm, amsfonts}
\usepackage{extsizes}
\usepackage{hyperref}
\usepackage{xcolor,graphicx}
\usepackage{hyperref}
\usepackage{svg}
\usepackage{diagbox}
\usepackage{graphicx,color}    

\usepackage{url} 

\newtheorem{theorem}{Theorem}[section]
\newtheorem{corollary}[theorem]{Corollary}
\newtheorem{proposition}[theorem]{Proposition}

\theoremstyle{definition}
\newtheorem{definition}[theorem]{Definition}
\newtheorem{example}[theorem]{Example}
\newtheorem{remark}[theorem]{Remark}

\def\r{\mathbb R}
\def\l{\mathbb L}

\begin{document}

\title{Newton's problem  of minimal resistance in Lorentz-Minkowski space }
\author{Rafael L\'opez}
\address{Department of Geometry and Topology\\ University of Granada. 18071 Granada, Spain}
\email{rcamino@ugr.es}
\begin{abstract}
We extend Newton's problem of minimal resistance to the Lorentz-Minkowski space. We derive the functional energy and  determine the Euler-Lagrange equation. In contrast to the Euclidean case, this equation is quasilinear elliptic, and thus, a maximum principle holds in this context. We obtain  the solutions of separable variables of this equation via separation of variables. Furthermore, we find all radial solutions to the problem, which present conical singularities at the origin.  We also analyze the Single Shock Condition. 
  \end{abstract}

\maketitle
\section{Statement of   Newton's problem in Lorentz-Minkowski space}

The purpose of this paper is to extend   Newton's problem of minimal resistance   to the Lorentz-Minkowski space $\l^3$. In Euclidean space $\r^3$, and proposed by Newton in  his {\it Principia Mathematica}, the problem consists of finding the shape of a solid body that experiences  minimal resistance while moving at a constant  velocity through a homogeneous fluid. The literature on this problem is extensive, and from a mathematical viewpoint, we refer the reader to the surveys   \cite{bu,bk} and the references therein. Variants of Newton's problem have been considered recently regarding the type of bodies, collisions, and fluid \cite{bfk,bk,cl,hk,lo2,pt,pt2}

The motivation for   extending   Newton's problem to $\l^3$ is that the metric of $\l^3$ is not Riemannian which implies important differences. The existence of vectors whose self-inner product may be   zero or negative is the reason for the interest in $\l^3$ in Special Relativity. 
 First, we provide the necessary  context following formally the same approach as in the classical problem. Consider the Lorentz-Minkowski space $\l^3$, which is the Euclidean space $\r^3$ with Cartesian coordinates $(x,y,z)$, and  endowed with the Lorentzian metric $\langle,\rangle=dx^2+dy^2-dz^2$.  

The assumptions of the problem in $\l^3$ are the same as in Euclidean space with the only difference being that the metric used to  measure angles is now the Lorentzian metric of $\l^3$. Thus, we assume that   the fluid is uniform, its particles are non-interacting,  move with the same velocity, and when they collide with the (external) surface $S$ of the body $\mathcal{B}$, they transfer their normal  momentum without friction.  In $\r^3$, the model satisfies the sine-squared pressure law, which requires the computation of angles between vectors. In $\l^3$, there is no a natural notion of angle between arbitrary vectors in general, except when that the vectors are timelike. Therefore, the   differences with respect to the Euclidean setting are as follows:
\begin{enumerate}
\item It is important to specify the direction of the fluid. In $\l^3$ there are spacelike, lightlike, and timelike directions. In our model, the chosen direction of the fluid is the $z$-axis (time coordinate in Relativity) which is of the timelike type. Let   $\vec{e}_3=(0,0,1)$. 
\item The  surface $S$  of     $\mathcal{B}$ is   spacelike. A spacelike surface in $\l^3$ is a surface whose induced metric is Riemannian. This implies that its normal vector $N$  is timelike. This allows us to well-define the angle between $N$ and the direction  $\vec{e}_3$ of the particles. 
\end{enumerate}

\begin{figure}[hbtp]
\begin{center}
\includegraphics{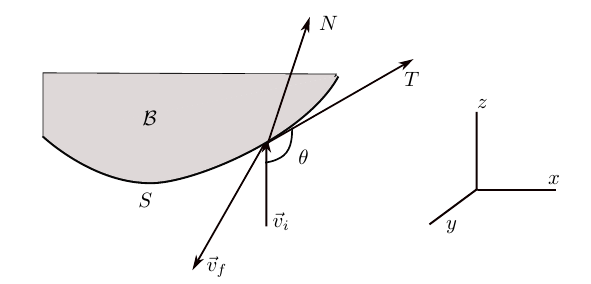}
  \end{center}
\caption{The model of the resistance problem in $\l^3$. The initial velocity of the particles  is $\vec{e}_3=-\sinh\theta T+\cosh\theta N$, and the final velocity is $\vec{v}_f=-\sinh\theta T-\cosh\theta N$.}\label{fig1}
\end{figure}

After   determining the direction in which the particles move, and the type of body they collide with, the  resistance problem in $\l^3$  can   be formulated as follows: {\it find the body whose boundary surface is spacelike and  offers minimum or maximum  resistance  to the motion within a fluid whose particles move with constant velocity in the direction $\vec{e}_3$}. 

We   specify the mathematical  framework of the problem (Fig. \ref{fig1}). Assume that $S$ is  the graph of a function $u$ defined on a domain $\Omega$ of $\r^2$, where we identify $\r^2$ with the $xy$-plane of $\r^3$. We assume that  the admissible functions belongs to the class of $C^1$  piecewise functions on $\overline{\Omega}$. Since $S$ is   spacelike, it follows that $|Du|<1$ in $\Omega$. The unit normal vector of $S$ is 
\begin{equation}\label{n}
N=\frac{1}{\sqrt{1-|Du|^2}}\left( Du,1\right).
\end{equation}
We   compute the resistance of $\mathcal{B}$ as it moves through  the fluid.  We assume that each  particle has unit mass and unit speed, and that the  velocity is $\vec{v}_i=\vec{e}_3$. Since $\vec{v}_i$ and $N$ are timelike and $N$ points upwards, there exists $\theta>0$ such that 
\begin{equation}\label{nv}
\langle\vec{v}_i,N\rangle=\langle \vec{e}_3,N\rangle=-\cosh\theta.
\end{equation}
The number $\theta$ is called the angle between $\vec{v}_i$ and $N$. Consider the plane containing $\vec{v}_i$ and the final velocity $\vec{v}_f$ of the particle after colliding with $S$. Let $\{\vec{e}_x,\vec{e}_3\}$ be a unit basis for this plane (Fig. \ref{fig1}). Suppose that the point where the particle collides with $S$ is $(x_0,y_0,u(x_0,y_0))$, where $\vec{e}_x=\frac{Du}{|Du|}$.  
Decomposing $\vec{v}_i$ into its tangential and normal components with respect to $S$, we have
 $$\vec{v}_i=-\sinh\theta T+\cosh\theta N.$$
  When a particle  collides  with $S$, the tangential component of its final velocity $\vec{v}_f$ remains unchanged, while  the normal component changes sign. Thus,
  $$\vec{v}_f=-\sinh\theta T-\cosh\theta N.$$
  It follows that $\vec{v}_f-\vec{v}_i=-2\cosh\theta N$.     The resistance contributed by a single  particle is the component of its momentum in the direction of $\vec{v}_i$, namely, 
 $$\langle \vec{v}_f-\vec{v}_i,\vec{v}_i\rangle=-2\cosh\theta \langle N,\vec{v}_i\rangle=2(\cosh\theta)^2.$$
 This identity can be seen as the model in $\l^3$ satisfying a {\it hyperbolic cosine-squared law}.  Using \eqref{n} and the relation \eqref{nv}, we obtain 
 $$(\cosh\theta)^2=\frac{1}{1-|Du|^2}.$$
 Up to this point,  the resistance problem is unconstrained, as  $u$ is arbitrary except for the spacelike condition $|Du|<1$. To  compute the total resistance of the body $\mathcal{B}$, we impose  the ``Single Shock Property'':   after colliding with  $S$, particles  do not strike the surface again.   Consequently,  up to constant factors, we have the following definition.
 
 \begin{definition}Let $S$ be a spacelike surface in $\l^3$ given by $z=u(x,y)$, where $u$ is a $C^1$ function defined in  a bounded domain $\Omega\subset\r^2$. The resistance of $S$ against a fluid   moving along the $z$-direction with constant velocity is 
 \begin{equation}\label{eq1}
 E[u]=\int_\Omega \frac{1}{1-|Du|^2}\, dxdy.
 \end{equation}
 \end{definition}
 
 Initially, the problem can be formulated as maximizing or minimizing the energy $E$. Recall that a similar situation is the problem of the area of a surface in $\r^3$ or in $\l^3$, whereas in $\r^3$, the problem is to minimize (minimal surfaces) but in $\l^3$ it is to maximize (maximal surfaces). 
 
 In order to find solutions of the Euler-Lagrange equation associated with $E$, it is natural to restrict the class of admissible functions.  In Section \ref{s3}, we prove that linear functions are the only  solutions to Newton's problem that can be expressed by separation of variables. Section \ref{s4} is devoted to analyzing the problem of maximizing the energy. In Section \ref{s5}, we assume that the solution is radial, finding explicit parametrizations of the solutions. In contrast to the Euclidean case, the radial solution can be defined up to the rotation axis,  with a conical-type singularity at the axis. Finally, we study the Single Shock Condition in the Lorentzian context (Section \ref{s6})
  
\section{Preliminaries and examples}\label{sec2}

In the Lorentz-Minkowski space, a vector $\vec{v}$ is said to be  spacelike (resp. timelike, lightlike) if $\langle \vec{v},\vec{v}\rangle>0$ (resp. $<0$, $=0$). The modulus of a vector $\vec{v}$ is $|\vec{v}|=\sqrt{|\langle \vec{v},\vec{v}\rangle}$. The set of   lightlike vectors is the lightlike cone $\{\vec{v}\in\l^3:\langle \vec{v},\vec{v}\rangle=0\}$.

Two timelike vectors $\vec{v}, \vec{w}\in\l^3$ cannot be orthogonal, i.e., $\langle \vec{v},\vec{w}\rangle\not=0$. We say that they are in the same timelike cone (or that they have the same timelike orientation) if $\langle \vec{v},\vec{w}\rangle<0$. This provides a timelike orientation on $\l^3$ for  all timelike vectors once a reference timelike vector is fixed in $\l^3$. We fix the timelike orientation induced by the vector $\vec{e}_3=(0,0,1)$.  The reverse Cauchy-Schwarz inequality for timelike vectors asserts that $|\vec{v}|^2|\vec{w}|^2<\langle \vec{v},\vec{w}\rangle^2$. If $\vec{v}$ and $\vec{w}$ have the same timelike orientation, the angle between $\vec{v}$ and $\vec{w}$  is defined  as the    unique $\theta\geq 0$ such that 
$$\cosh\theta=-\frac{\langle \vec{v},\vec{w}\rangle}{|\vec{v}||\vec{w}|}.$$
See \cite{lo}. An   immersion $\Psi\colon S\to\l^3$ of a surface $S$ is called spacelike if the induced metric endowed from $\l^3$ is Riemannian at all   points of $S$. If $S$ is the graph of a function $u$, this is equivalent to $|Du|<1$, or equivalently, that  the unit normal vector   given in \eqref{n} is timelike. Note that  $N$ has the same timelike orientation as $\vec{e}_3$., and  the angle $\theta$ between $N$ and $\vec{e}_3$ at each point of $S$ is given by   
$$\cosh\theta=\frac{1}{\sqrt{1-|Du|^2}}.$$
An interesting surface   in $\l^3$ is the lightlike cone, which is defined as any translation of $\mathcal{K}=\{(x,y,z)\in\l^3\colon x^2+y^2=z^2\}$, as well as, the upper lightlike cone $\mathcal{K}^+=\mathcal{K}\cap\{z> 0\}$ and the lower lightlike cone $\mathcal{K}^-=\mathcal{K}\cap\{z< 0\}$. The   induced metric on $\mathcal{K}$ is degenerate. 
 If the vertex is situated at a point $p$, we denote these by  $\mathcal{K}_p$,   $\mathcal{K}^+_{p}$ and $\mathcal{K}^-_{p}$, respectively.

 We calculate the resistance of some spacelike surfaces. The following examples are   surfaces of revolution obtained by rotating a planar curve $u=u(r)$ contained in the $xz$-plane  about the $z$-axis, where $u\colon [0,R]\to\r$. The domain of $u(x,y)=u(\sqrt{x^2+y^2})$ is the disk  $B_R\subset\r^2$  of radius $R>0$ centered at the origin of $\r^2$.  In polar coordinates, the resistance of $u$  is
 \begin{equation}\label{eq2}
 E[u]=2\pi\int_{0}^{R}\frac{r}{1-u'(r)^2}\, dr.
 \end{equation}
In all examples, we will take the boundary condition   $u(R)=0$. 
\begin{enumerate}
\item Horizontal disks. Let  $u(x,y)=0$ in $B_R$. Then 
 $E[u]=\pi R^2$.

\item  Hyperbolic planes. A hyperbolic plane of radius $\rho>0$ centered at the origin of $\l^3$ is defined by  $x^2+y^2-z^2=-\rho^2$. These surfaces,with two components,  are spacelike and play the role of spheres in $\l^3$ because their Gaussian curvature  is constant.  Consider   the upper component    
\begin{equation}\label{hj}
 u(r)=-\sqrt{\rho^2+R^2}+\sqrt{\rho^2+r^2},
\end{equation}
which is called a hyperbolic cap of radius $\rho$.  Then, $u(R)=0$, and  
$$E[u] = \frac{\pi  R^2 \left(2 \rho^2+R^2\right)}{2 \rho^2}.$$
Note that, for a fixed domain $B_R$, we have
$$\lim_{\rho\to 0}E[u]=\infty,\quad \lim_{\rho\to \infty}E[u]= \pi R^2.$$
Thus, as $\rho\to 0$, the hyperbolic cap  \eqref{hj} tends to   the upper lightlike cone $\mathcal{K}^+_{(0,0,-R)}$. 

Here, we establish a   relation between the resistance of a disk $B_R$ and a hyperbolic cap in order to compare them with the Euclidean case. Recall that Newton obtained that the ratio between the resistances of a sphere and a disk of the same radius is $1/2$. From the above computations,   the analogous result in $\l^3$ gives 
$$\frac{\mbox{resistance of hyperbolic cap of radius $R$}}{\mbox{resistance of a disk of radius $R$}}=\frac{3}{2}.$$

\item Spacelike cones.  Consider a cone $z^2=\lambda^2(x^2+y^2)$, where $\lambda\in (0,1)$. The parameter $\lambda$ represents the slope of the cone,  and the condition $0<\lambda<1$ ensures  that the cone is   spacelike. Taking  the upper part of the cone 
\begin{equation}\label{cj}
 u(r)=-\lambda R+\lambda r,\quad r\in [0,R],
\end{equation}
we  have $u(R)=0$. Then 
$$E[u]= \frac{\pi R^2}{1-\lambda^2}.$$
Consequently,
$$\lim_{\lambda\to 0}E[u]=\pi R^2,\quad \lim_{\lambda\to 1}E[u]= \infty.$$
For  a fixed $R$,  if $\lambda \to 1$, the spacelike cone $u$ tends to $\mathcal{K}^+_{(0,0,-R)}$.
 
 \end{enumerate}

It is clear that the study of extremals (or critical points) of the energy $E$ is invariant under  rigid motions of $\l^3$ and under dilations.

\begin{proposition} \label{pr-22}
Let $\Omega\subset\r^2$ be a convex bounded domain. The problem of finding extremals of the energy $E$ is invariant under  dilations with center $O\in\Omega$.
\end{proposition}

\begin{proof} 
Let  $c>0$ and suppose $O=(0,0)$. Given a function $u=u(x,y)$, by the convexity of $\Omega$, we define 
$$v\colon c\cdot \Omega\to\r, \quad v(x,y)=c u(\frac{1}{c}(x,y)).$$
Let $\Phi\colon\r^2\to\r^2$ be the dilation defined by $\Phi(x,y)=c(x,y)$. Note that $v\circ\Phi=cu$ and $|Dv|^2=|Du|^2\circ \Phi^{-1}$. By the  change of variables formula, and since $\mbox{Jac}(\Phi)=c^2$, we have 
\begin{equation*}
E[v]=\int_{c\Omega}\frac{1}{1-|Dv|^2}dxdy =
\int_\Omega\frac{c^2}{1-|Du|^2}\, dxdy=c^2 E[u].
\end{equation*}
Since the energy $E$ changes by a constant factor, the extremals are the same regardless of the value of $c$.
\end{proof}
\section{The Euler-Lagrange equation}\label{s3}

We calculate the Euler-Lagrange equation for the functional $E$. We assume that $u$ is of class$C^2$ in its domain. Since the Lagrangian  of $E$ does not depend on $u$, the Euler-Lagrange equation is 
$$\mbox{div}\frac{Du}{(1-|Du|^2)^2}=0.$$
Expanding this expression gives the quasilinear equation 
\begin{equation}\label{eqs}
(1-|Du|^2)\Delta u+4 D^2u(Du,Du)=0,
\end{equation}
or equivalently, 
\begin{equation}\label{eqs2}
(1+3u_x^2-u_y^2)u_{xx}+(1-u_x^2+3u_y^2)u_{yy}+8 u_x u_y u_{xy}=0.
\end{equation}
  If $u$ is a solution to \eqref{eqs2}, we say that $u$ is an extremal of $E$, and its graph   $\{(x,y,u(x,y))\colon (x,y)\in\Omega\}$  is called a {\it stationary surface}. The simplest examples of  stationary surfaces are spacelike planes $z= u(x,y)=ax+bu+c$, where $a^2+b^2<1$ (spacelike condition). A major difference with the Euclidean case is that \eqref{eqs2} is an elliptic equation.

\begin{proposition}\label{pr-31}
 The Euler-Lagrange equation of the functional $E$ is of quasilinear elliptic type.
\end{proposition}
\begin{proof} It suffices to compute the eigenvalues of the second-order terms of \eqref{eqs2}, which yields  
$$\lambda=1-|Du|^2,\quad \Lambda=1+3|Du|^2.$$
Then $0<\lambda\leq \Lambda$ due to the spacelike condition $|Du|<1$. 
\end{proof}
This result contrasts with  its counterpart in Euclidean space, which we unify in the following presentation. Let $\epsilon\in\{-1,1\}$ and define the functional 
\begin{equation}\label{e1}
E_\epsilon[u]=\int_\Omega\frac{1}{1+\epsilon|Du|^2},\qquad \begin{array}{ll}
\epsilon=1,& \mbox{Euclidean case},\\
\epsilon=-1,& \mbox{Lorentzian case}.
\end{array}\end{equation}
The Euler-Lagrange equation of $E_\epsilon$ is 
\begin{equation}\label{eqs3}
(1+\epsilon|Du|^2)\Delta u-4\epsilon D^2u(Du,Du)=0,
\end{equation}
and the eigenvalues of the second order terms are 
\begin{equation*}
\lambda_\epsilon=1-(2+\epsilon)|Du|^2,\qquad 
\Lambda_\epsilon=1+(2-\epsilon)|Du|^2.
\end{equation*}
In $\r^3$, the equation \eqref{eqs3} is of  elliptic-hyperbolic  type, as indicated in \cite{bfk}.  

Coming back to $\l^3$, and as a consequence of Prop. \ref{pr-31}, the maximum principle holds, as well as, the uniqueness of the solutions to the Dirichlet problem (\cite[Sect. 10]{gt}). From the examples of spacelike planes, we deduce the following  uniqueness result.

\begin{corollary} \label{c-32}
Let $u\in C^2(\Omega)\cap C(\overline{\Omega})$ be a solution to \eqref{eqs2} defined on a bounded domain $\Omega$.
\begin{enumerate}
\item If $u=c$ along $\partial\Omega$,  then $u=c$ on $\Omega$.
\item Suppose that $\Omega$ is doubly-connected with $\partial\Omega=\Gamma_1\cup\Gamma_2$. If $u=c_i$ along $\Gamma_i$, $i=1,2$,  and $c_1\leq c_2$, then $c_1\leq u\leq c_2$ in $\Omega$.
\item If $u$ attains an interior local maximum or minimum, then $u$ is constant in $u$.
\end{enumerate}  
\end{corollary}


As we stated, and without   restriction on the class of admissible functions, global minimizers of $E$ are obtained when $u$ is a constant function. The corresponding stationary surfaces are open sets of horizontal planes, where   $E[u]=\mbox{area}(\Omega)$. If we impose constant boundary conditions, $u=c$, and  if $u$ is not constant, then   $u$ is  necessarily not $C^2$ at some points of $\Omega$. It is expected that $|Du|\to 1$ at the points where regularity is lost because the coefficient of $\Delta u$ in \eqref{eqs} tends to $0$.  Thus, at these points, the surface is tangent to a cone $\mathcal{K}$.

We finish this section by obtaining all stationary surfaces given by separation of variables. 

\begin{theorem}\label{t23}
 Linear functions are the only extremals of the form $u(x,y)=f(x)+g(y)$. In consequence, spacelike planes are the only  stationary surfaces given by separation of variables.
\end{theorem}

\begin{proof} Suppose that $u(x,y)=f(x)+g(y)$ is an extremal.  Let $I$ and $J$ be the domains of $f$ and $g$, respectively. 
Substituting this into \eqref{eqs2} gives 
\begin{equation}\label{fg}
(1-f'^2-g'^2)(f''+g'')+4(f'^2f''+g'^2 g'')=0.
\end{equation}
Differentiating first with respect to $x$, and then with respect to $y$, we obtain
\begin{equation}\label{fg2}
g'g''f'''=-f'f'' g'''.
\end{equation}
Suppose that there is $(x_0,y_0)$ such that \begin{equation}\label{f3}
(f'^2)'(x_0)(g'^2)'(y_0)\not=0,
\end{equation}
from which we will arrive at a contradiction.  Then, there is a constant $c\in\r$ such that 
$$\frac{f'''}{f'f''}=c=-\frac{g'''}{g'g''}$$
in an open set of $I\times J$ around $(x_0,y_0)$. 
We consider two cases:
\begin{enumerate}
\item Case $c=0$. Then $f'''=g'''=0$, which implies there exist constants $a_i,b_i$ ($1\leq i\leq 3$) such that
$$f(x)=a_1\frac{x^2}{2}+a_2x+a_3,\quad g(y)=b_1\frac{y^2}{2}+b_2y+b_3.$$
Moreover, $a_1b_1\not=0$ by \eqref{f3}.   Substituting these expressions into \eqref{fg} yields a polynomial equation in $x$ and $y$, whose  coefficients must vanish. The coefficients of $x^2$ and $y^2$ are $a_1^3(3a_1-b_1)$ and $b_1^2(3b_1-a_1)$, respectively. This implies $a_1=b_1=0$,   a contradiction. 
\item Case $c\not=0$. From \eqref{fg2},  there are constants $a_1, b_1\in\r$ such that 
$$f''=\frac{c}{2}f'^2+a_1,\quad g''=-\frac{c}{2}g'^2+b_1.$$
Substituting these into \eqref{fg} gives
$$ (6 a_1-2 b_1+c) f'^2-(2 a_1-6 b_1+c) g'^2+2 (a_1+b_1)+3 c f'^4-3 c g'^4 =0.$$
Viewing  this identity  as a polynomial equation in $f'$, all coefficients must vanish because $f'\not=0$. We see that the coefficient of $f'^4$ is  $3c$. This yields $c=0$, a contradiction. 
\end{enumerate}

As a consequence of these arguments, and from \eqref{f3}, we must have $(f'^2)' (g'^2)' =0$   identically in $I\times J$. Therefore, either $(f'^2)'=0$ or $(g'^2)'=0$. Without loss of generality,   assume $(f'^2)'=0$; hence, there is  $a\in\r$ such that $f(x)=ax+b$. Substituting this into \eqref{fg}, we obtain
$$(1-a^2+3g'^2)g''=0.$$
This implies $g''=0$ or $1-a^2+3g'^2=0$. In both cases, we conclude that $g$ is linear. This implies that both $f$ and $g$ are linear functions,  completing  the proof.
\end{proof}

The same argument proves the following result in the Euclidean space.

\begin{theorem}  Planes are the only stationary surfaces in $\r^3$ for the energy $E_1$, \eqref{e1}, of the form $u(x,y)=f(x)+g(y)$.  
\end{theorem}

\section{Maximizers of the energy}\label{s4}

In this section, we consider the problem of finding maximizers of the energy $E$. Note that, in principle, this problem is meaningful when recalling that in $\l^3$, for the area energy of a spacelike surface, the relevant problem is to find maximizers. 
Returning to the examples of cones and hyperbolic caps, we observe that one can construct sequences of  spacelike surfaces whose resistance diverges to infinity. 

Let 
 $$\mathcal{C}(\Omega)=\{u\colon\overline{\Omega}\to\r\colon \mbox{$u$ is  $C^1$-piecewise},  |Du|^2<1\}.$$

\begin{proposition}\label{pr-u}
The unconstrained maximum resistance problem has no solution in the class of functions $\mathcal{C}(B_R)$.
\end{proposition}

\begin{proof} 
Let $u_\lambda$ be the function defined in \eqref{cj}. Then $u_\lambda\in \mathcal{C}(B_R)$, and $\lim_{\lambda\to 1}E[u_\lambda]=\infty$.
\end{proof}

One can also consider the hyperbolic caps   defined in \eqref{hj}. As $\rho\to 0$, the resistance tends to infinity.  

Proposition \ref{pr-u} also holds for any bounded domain $\Omega$ in $\r^2$ because it suffices to consider a small disk $B_R$ contained in $\Omega$.

In the Euclidean case, it is  natural to restrict the admissible functions to be bounded, satisfying  $0\leq u\leq M$ for some  constant $M$. Even within   this family, the infimum of the resistance is $0$. A difference from the Euclidean case is that  in $\l^3$ this   boundedness assumption must be refined, because the spacelike condition $|Du|<1$ imposes restrictions on the height of the function $u$. 

\begin{proposition} 
Let $\Omega\subset\r^2$ be a bounded domain. Let $u\in C^1(\overline{\Omega})$ be a function such that $u=0$ on $\partial\Omega$. If $|Du|<1$ on $\overline{\Omega}$, then 
$$|u(p)|<\mbox{dist}(p,\partial\Omega),\quad p\in\Omega.$$
In particular, if $\Omega=B_R$, then $|u|< R$ in $\overline{B_R}$.
\end{proposition}

\begin{proof}
Let $p\in\Omega$ and  let $q\in\partial\Omega$ be such that $\mbox{dist}(p,\partial\Omega)=d(p,q)$. Since the segment joining $p$ and $q$ is included in $\overline{\Omega}$, we define the function $f(t)=u(\alpha(t))$, where $\alpha(t)=(1-t)q+tp$, $t\in [0,1]$. By the mean value theorem, $f(1)-f(0)=f'(t_0)$ for some $t_0\in[0,1]$. Then, we have 
$$|u(p)|=|f(1)-f(0)|=|\langle Du(\alpha(t_0)),\alpha'(t_0)\rangle<|\alpha'(t_0)|=\mbox{dist}(p,\partial\Omega).$$
\end{proof}

If $\Omega=B_R$, we need to  restrict the class to functions that do not attain the value $R$.  Given $M$, with $0<M<R$, define 
$$\mathcal{C}_M(B_R)=\{u\in \mathcal{C}(B_R)\colon   0\leq u\leq M, |Du|^2<1\}.$$
If $u$ attains a maximum of $E$ in $\mathcal{C}_M(B_R)$, then we prove that $u$ necessarily satisfies the boundary conditions $u(0)=0$ and $u(R)=M$. 

\begin{proposition} Let $u\in \mathcal{C}_M(B_R)$ be a function that attains a local maximum of $E$ in the class of functions $\mathcal{C}_M(B_R)$. Then, up to a reflection across a horizontal plane of $\r^3$, we have 
$$u(0)=0,\quad u(R)=M.$$
\end{proposition}

\begin{proof}
After a reflection across the plane $z=0$,  we can assume that $u(0)<u(R)$. Suppose that $u\in \mathcal{C}_M(B_R)$ such that $u(0)>0$. For $\epsilon>0$, define  the function 
$$v_\epsilon(r)=(1+\epsilon )(u(r)-u(0)).$$
It is clear that $v_\epsilon(0)=0$. We prove that, by choosing a suitable $\epsilon>0$, we have $v_\epsilon\in \mathcal{C}_M(B_R)$,  and 
$E[v_\epsilon]>E[u]$. Let $\delta=\frac{u(R)-u(0)}{M}<1$ and $s_M=\max_{r\in [0,R]}u'(r)<1$. Then $v_\epsilon$ is increasing with $v'_\epsilon=(1+\epsilon)u'(r)\leq (1+\epsilon)s_M$. The spacelike condition for $v_\epsilon$ requires $1+\epsilon<\frac{1}{s_M}$. Moreover, $v_\epsilon(R)=(1+\epsilon)\delta M$. Thus, we also require $1+\epsilon\leq \frac{1}{\delta}$. For any $\epsilon>0$ satisfying both conditions, we have $v_\epsilon\in \mathcal{C}_M(B_R)$ with $v_\epsilon(0)=0$ and 
$E[v_\epsilon]>E[u]$. This leads to a contradiction, proving $u(0)=0$.

Similarly, if $u\in \mathcal{C}_M(B_R)$ is such that $u(R)<M$, we can define, for $\epsilon>0$, the function 
$$w_\epsilon(r)=(1+\epsilon)(u(r)-u(R))+M.$$
Then, with the same notation for $\delta$ and $s_M$, we have $w_\epsilon(R)=M$, $w_\epsilon\in \mathcal{C}_M(B_R)$ and $E[w_\epsilon]>E[u]$. This would imply that $u$ is not a local maximizer, a contradiction. This proves that $u(R)=M$.
\end{proof}

\begin{proposition} \label{pr14}
The maximum of the functional $E$ in $\mathcal{C}_M(B_R)$ is infinite.
\end{proposition}
\begin{proof}
For each $n\in\mathbb{N}$, define $a_n=M+\frac{R-M}{n}$ and consider the $C^1$-piecewise  function 
$$u_n(r)=\left\{\begin{array}{ll}
\frac{M}{a_n}r,& 0\leq r\leq a_n,\\
M,& a_n\leq r\leq R.
\end{array}
\right.$$
Then $u_n$ is spacelike and $0\leq u_n\leq M$ in its domain. Observe that $u_n$ describes a spacelike cone that tends  to the lightlike cone $\mathcal{K}^+$ in $(0,M)$  as $n\to\infty$. We have  
\begin{equation*}
\begin{split}E[u_n]&=2\pi\int_0^R\frac{r}{1-u_n'^2}\, dr=2\pi\frac{a_n^2}{a_n^2-M^2}\int_0^{a_n}r\, dr=\pi\frac{a_n^4}{a_n^2-M^2}\\
&=-\frac{(M (n-1)+R)^4}{2 n^2 (M-R) (M (2 n-1)+R)}.
\end{split}
\end{equation*}
Hence $\lim_{n\to\infty}E[u_n]=\infty$. 
\end{proof}

\begin{corollary} Let $\Omega$ be a domain and   $R>0$   such that   $B_R\subset\Omega$. If $0<M<R$, define 
$$\mathcal{C}_M(\Omega)=\{u\in C^1(\Omega)\colon   0\leq u\leq M\}.$$
Then the maximum of the functional $E$ on $\mathcal{C}_M(\Omega)$ is infinite.
\end{corollary}

\begin{proof}
Since $\mathcal{C}_M(B_R)\subset \mathcal{C}_M(\Omega)$, the result follows immediately from Prop. \ref{pr14}.
\end{proof}

\begin{example} We can find   examples of  hyperbolic caps where the energy diverges. Indeed, let $a_n$ be as in the previous proof. Let 
$$\rho_n=\frac{a_n^2-M^2}{2M},$$
and $a_n=M+\frac{R-M}{n}$. Define 
 $$u_{n}(r)=\left\{\begin{array}{ll} -\rho_n+\sqrt{\rho_n^2+r^2},& 0\leq r\leq a_n\\
M,& a_n\leq r\leq  R.
\end{array}\right.$$
Then $u_n(0)=0$, $u_n(a_n)=M$ and $0\leq u_n\leq M$.  The computation of the resistance gives 
$$E[u_n]=\frac{\pi a_n^2}{2\rho_n^2}(a_n^2+2\rho_n^2)\to \infty$$
as $n\to 0$, because $\rho_n\to 0$ and $a_n\to M$.
Note that if $n\to\infty$, then $u_n$ tends to the lightlike cone $\mathcal{K}^+$ in the interval $[0,M]$.
\end{example}

\section{The radial case}\label{s5}

We study the radial solutions of the Euler-Lagrange equation \eqref{eqs}, or equivalently,   stationary surfaces  which are surfaces of revolution about the $z$-axis. Since $S$ is spacelike, then $S$ is a locally graph over $\r^2$. If, for instance, we are looking for rotational spacelike surfaces defined on $B_R$, then $u$ is a global graph on $B_R$, and thus, $u$ can be written as $u=u(r)$. The surface $S$   can be described by its generating curve $\gamma(r)=(r,0,u(r))$, where 
$r\in [a,b]$, with  $a\geq 0$. Hence, $S=\{(r\cos t, r\sin t,u(r))\colon a\leq r\leq b,t\in\r\}$. The Lagrangian of the energy $E$ is 
\begin{equation}\label{ff}
F(r,u,u')=\frac{r}{1-u'^2}.
\end{equation}
Since $F$ does not depend explicitly on $u$, the Euler-Lagrange equation is  
$$\frac{d}{dr}\frac{ru'(r)}{(1-u'(r)^2)^2}=0.$$
Therefore, there exists  a constant $c_1\in\r$ such that 
\begin{equation}\label{r1}
\frac{ru'(r)}{(1-u'(r)^2)^2}=c_1.
\end{equation}
 If $c_1=0$, then $u(r)=c$ for some constant $c\in\r$, and $S$ is a horizontal spacelike plane. Suppose now $c_1\not=0$. Without loss of generality, from now on we assume $c_1>0$.  Equivalently, we have $u'>0$ in its entire domain due to \eqref{r1}. This proves that $u(r)$ is strictly increasing in its domain. Let $p=u'$ be the new independent variable.   Then \eqref{r1} becomes
\begin{equation}\label{r2}
r= \frac{c_1}{p}(1-p^2)^2.
\end{equation}
In particular, $r(p)$ is a decreasing function of $p$. 
We now find a parametric representation of $\gamma$ in terms of the variable $p$. Differentiating   \eqref{r2} gives 
$$\frac{dr}{dp}=c_1 (-\frac{1}{p^2}-2+3p^2).$$
Then
$$u=\int p\, dr= \int c_1 p\left(-\frac{1}{p^2}-2+3p^2\right)dp=c_2- c_1 \left(\log(p)+p^2-\frac{3p^4}{4}\right),$$
where we use  that fact that $p>0$. 
Thus, we  have $\gamma(p)=(r(p),u(p))$, where 
\begin{eqnarray}
r(p)&=&\frac{c_1}{p}(1-p^2)^2,\label{p1}\\
u(p)&=&c_2-c_1\left(\log(p)+p^2-\frac{3p^4}{4}\right).\label{p2}
\end{eqnarray}
We study the existence of radial extremals of $E$ parametrized by \eqref{p1}-\eqref{p2}. The first examples are horizontal planar disks $u(r)=c$, $c\in\r$. 

We now suppose that $u$ is not constant. Since at $r=R$ the function is constant, and if we are looking for radial solutions defined on $B_R$, then Cor. \ref{c-32} implies that $u(r)$  must lose  $C^2$-regularity at some point of $(0,R)$. We will see that, indeed, the solution $u(r)$ presents a singularity ($u'=1$) at  $r=0$. The following result proves the existence of rotational stationary surfaces defined in punctured disks around the origin. We also give  a full description of the geometric properties of the radial extremals.  

\begin{theorem}\label{t51}
 Let $M$ and $R$ be positive numbers with $M<R$. Then there is a unique extremal of $E$ of radial type, $u=u(r)$, given as a solution $u \colon (0,R]\to\r$  of \eqref{r1} such that 
\begin{equation}\label{c1}\lim_{r\to 0}u(r)=0,\quad u(R)=M.
\end{equation}
Moreover: 
\begin{enumerate}
\item  The function $u$ is strictly increasing and $0< u(r)\leq M$ for all $r\in (0,R]$.
\item The function $u$ can be extended to   $r=0$, with the property 
$$\lim_{r\to 0}u'(r)=1.$$
\item The curve $\gamma $ is regular   on its  domain.
\item The function $u(r)$ can be extended to $(0,\infty)$ while maintaining  the spacelike condition $u'^2<1$. Moreover, 
$$\lim_{r\to \infty}u(r)=\infty,\quad  \lim_{r\to \infty}u'(r)=0.$$
\item The function $u(r)$ is concave. 
\end{enumerate} 
\end{theorem}

\begin{proof}
Based on \eqref{r1} and \eqref{r2}, the function $u$ can be defined with the variable $p=u'$ according to \eqref{p1} and \eqref{p2}. We find the values of $c_1$ and $c_2$ that satisfy the boundary conditions \eqref{c1}. 

From \eqref{p1}, $u(r)$ is defined up to  $r=0$ as $p\to 1$. If we require $u(0)=0$,  from \eqref{p2} we deduce $c_2=c_1/4$. Replacing this in \eqref{p2}, we obtain 
\begin{equation}\label{p3}
u(p)=c_1\left(\frac14-\log(p)-p^2+\frac{3p^4}{4}\right).
\end{equation}
We now find the value of $p$ such that $u(R)=M$, which  is the second condition in \eqref{c1}. The function   $r\colon (0,1]\to\r$, $r=r(p)$, is  strictly decreasing with 
$$\lim_{p\to 0}r(p)=\infty,\quad \lim_{p\to 1}r(p)=0.$$
Thus,   the range of $r(p)$ is $[0,\infty)$. This proves that $u$ can be extended to $(0,\infty)$ satisfying the condition $u'(r)<1$ (item 4). From \eqref{p1}, let $p_R$ be the unique value such that 
$r(p_R)=R$. This yields 
\begin{equation}\label{c111}
c_1=R\frac{p_R}{(1-p_R^2)^2},
\end{equation}
where the value of $p_R$ needs to be determined.  Now \eqref{p3} becomes at $p=p_R$, 
$$u(p_R)=\frac{  R\, p_R}{(1-p_R^2)^2}\left(\frac14-\log(p_R)-p_R^2+\frac{3p_R^4}{4}\right).$$
As we require   $u(p_R)=M$, then  
\begin{equation}\label{c2}
\frac{p_R}{(1-p_R^2)^2}\left(\frac14-\log(p_R)-p_R^2+\frac{3p_R^4}{4}\right)=\frac{M}{R}.
\end{equation}
Define the function  
$$g\colon (0,1)\to\r,\quad g(x)=\frac{x}{(1-x^2)^2}\left(\frac14-\log(x)-x^2+\frac{3x^4}{4}\right).$$
Thus, the condition \eqref{c2} is fulfilled if we find a solution to 
\begin{equation}\label{c22}
g(x)=\frac{M}{R},\end{equation}
where we know $\frac{M}{R}\in (0,1)$.
The function $g(x)$ is strictly increasing with 
$$\lim_{x\to 0}g(x)=0,\quad \lim_{x\to 1}g(x)=1.$$
 This proves the existence of a unique solution  $p_R$ to \eqref{c22}.  Moreover, since $p=u'>0$, the function $u(r)$ is strictly increasing, proving the item (1).

 We prove (2). By the chain rule,  
$$\lim_{r\to 0}u'(r)=\lim_{p\to 1}\frac{u'(p)}{r'(p)}=\lim_{p\to 1}p= 1.$$

We study the regularity of $\gamma$. From \eqref{p1} and \eqref{p2}, we have   
\begin{equation}\label{p5}
\begin{split}
r'(p)&=-c_1(\frac{1}{p^2}+2-3p^2),\\
u'(p)&=-c_1p(\frac{1}{p^2}+2-3p^2).
\end{split}
\end{equation}
If   $\gamma'(p)=(0,0)$ for $p\not=0$,   the   second equation of \eqref{p5} gives  
\begin{equation}\label{p4}
\frac{1}{p^2}+2-3p^2=0.
\end{equation}
 This yields $p^2=1$. In such a case, $r(p)=0$, which  is outside the domain of $u$.   This proves that $\gamma$ is regular everywhere.

For the item (4), consider the expression for $u(p)$ given in \eqref{p3}. The function $r'(p)$ in \eqref{p5} is negative because $c_1>0$ and $1+2p^2-3p^4\geq 0$ if $p^2\leq 1$. Taking the variable $p$   for $u$, we have $u\colon [p_R,1)\to\r$. Since $r(p)$ in \eqref{p1} is defined for all $p\in (0,1]$, then $u$ can be extended up to $(0,1]$. Moreover, 
$$\lim_{r\to\infty}u(r)= \lim_{p\to 0}u(p)=\infty,$$
$$\lim_{r\to\infty}u'(r)= \lim_{p\to 0}\frac{u'(p)}{r'(p)}=0.$$ 

Finally, we prove that $u$ is concave, that is, $u''(r)\leq 0$ for all $r$. Using the chain rule, and the expressions for $r(p)$ and $u(p)$ in \eqref{p1} and \eqref{p2}, we have 
$$u''(r)=\frac{r'(p)u''(p)-r''(p)u'(p)}{r'(p)^3}=\frac{p^6}{c1}(-1-2p^2+3p^4)\leq 0$$
for all $p\in (0,1]$. This completes the proof.

\end{proof}

Since at $r=0$, the slope of $u$ is the same as $\mathcal{K}^+$, we say that at $r=0$ the surface has a conical singularity. 

\begin{corollary} Up to a reflection across a horizontal plane and a translation in $\l^3$, non-planar axisymmetric solutions  to Newton's problem in $\l^3$ about the $z$-axis  can be extended to be entire graphs on $\r^2$. Moreover,  these solutions present a unique conical singularity at the origin, that is, the solution is tangent to the lightlike cone $\mathcal{K}_0^+$.
\end{corollary}

 \begin{figure}[hbtp]
\begin{center}
\includegraphics[width=.45\textwidth]{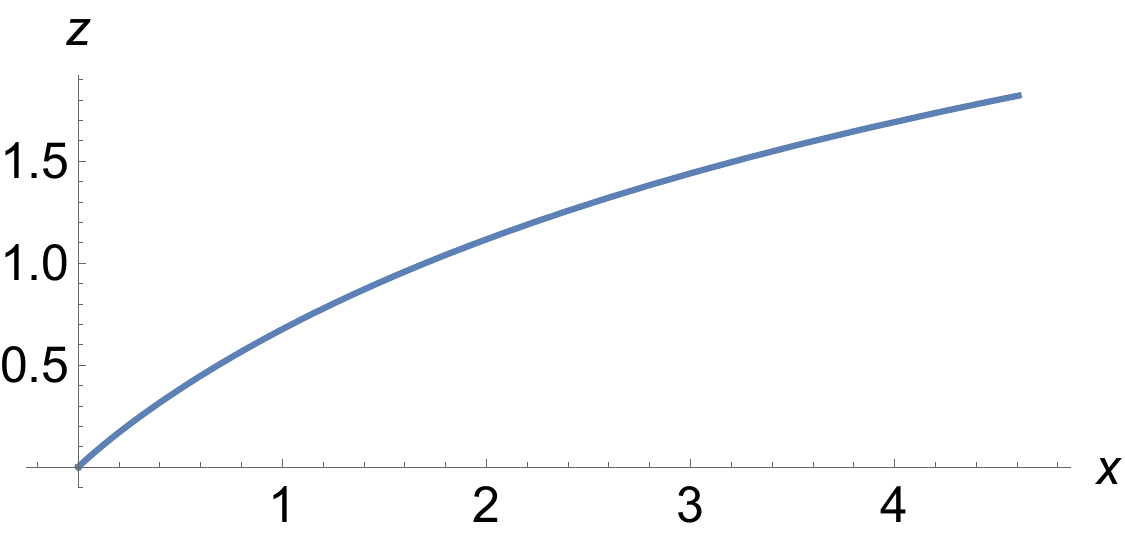}, \includegraphics[width=.5\textwidth]{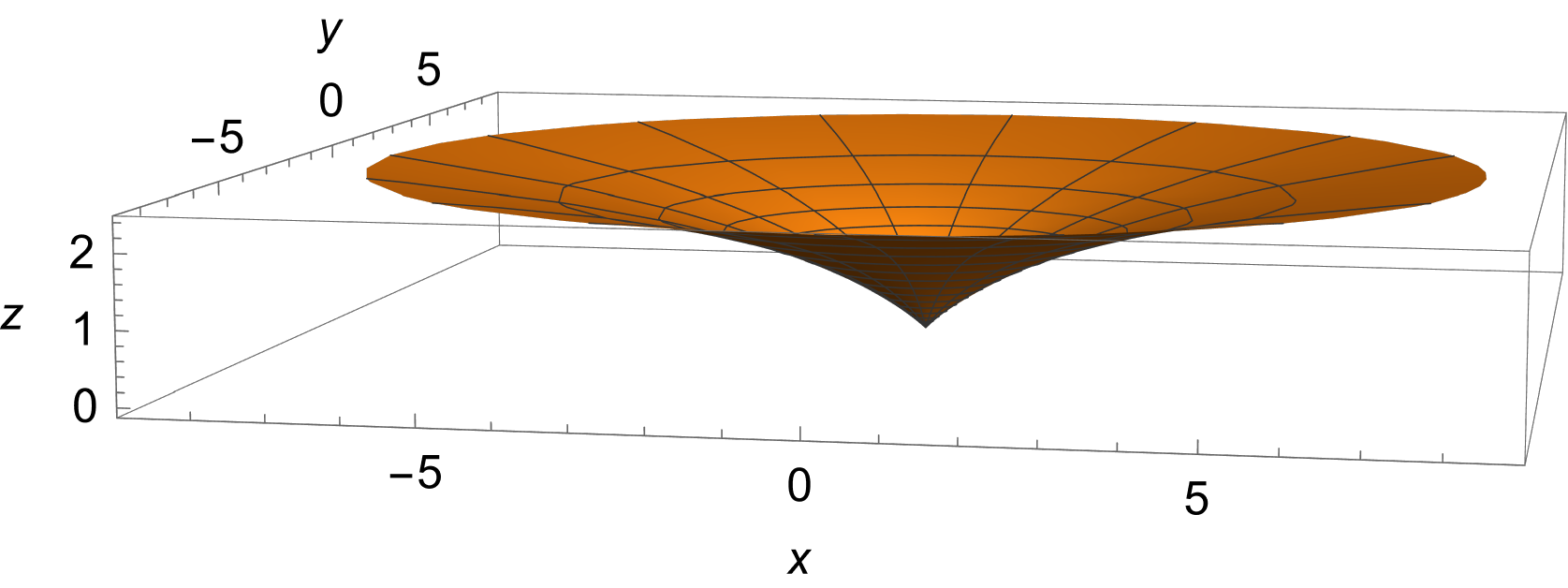}
\end{center}
\caption{Axisymmetric solutions of Newton's problem in $\l^3$: generating curve (left) and rotational stationary surface (right). }\label{fig2}
\end{figure}

\begin{remark} It is worth pointing  to point out a major difference between  the radial extremals of $\r^3$ and $\l^3$. In $\r^3$, radial extremals   cannot reach the rotation axis unless in the trivial case where $u$ is a constant function. In $\l^3$, and besides horizontal planes, radial extremals are defined in punctured disks, and at the origin, the slope of the extremal is $1$.
\end{remark}

We prove that the radial solutions are strong minima in the class of radial functions. Let 
$$C_r([0,R])=\{v\in C^2(0,R)\colon  \lim_{r\to 0}v(r)=0,\,  v(R)=M,\, v'^2<1\}.$$

\begin{theorem} Let $M$ and $R$ be positive numbers such that $M<R$. Let $u=u(r)$ be the radial solution defined in Thm. \ref{t51} with boundary conditions \eqref{c1}. Then $u$ is the strong minimum in  $C_r([0,R])$.
\end{theorem}

\begin{proof} 
  Note that   the   Legendre necessary condition to be a minimum is satisfied for the Lagrangian \eqref{ff} because  
$$F_{pp}(r,u,p)=2r\frac{1+3p^2}{(1-p^2)^3}>0.$$
 The region of  admissible elements $(r,u,u')$, where $u'^2<1$ has the property that if $(r,u.u')$ and $(r,u,v')$ are admissible, then  so is every $(r,u,(1-t)u'+tv')$ for all $t\in [0,1]$.

  We   compute the Weierstrass function $\mathcal{E}$ in the class $C_r([0,R])$. Let $u(r)$ be  the solution given in Thm. \ref{t51}. For a function $v\in C_r([0,R])$, we have
\begin{equation*}
\begin{split}
\mathcal{E}(r,u,u',v')&=F(r,u,v')-F(r,u,u')-(v'-u')\frac{\partial F(r,u,u')}{\partial u'}\\
&=\frac{r (u'-v')^2 \left(u'^2+2 u' v'+1\right)}{\left(1-u'^2\right)^2 \left(1-v'^2\right)}\\
&=\frac{r (u'-v')^2((u'+v')^2+(1-v'^2))}{\left(1-u'^2\right)^2 \left(1-v'^2\right)}.
\end{split}
\end{equation*}
Then $\mathcal{E}(r,u,u',v')>0$ for all $r>0$ because $v'^2<1$. This is sufficient to ensure that $u$ is a strong minimizer. 
\end{proof}

We finish this section by computing   the resistance of a radial solution $u(r)$ from Thm. \ref{t51} numerically. With the change of variable $p=u'(r)$, and taking into account  the value of $c_1$ in \eqref{c111},  we have 
\begin{equation}\label{ee}
\begin{split}
E[u]&=2\pi\int_0^R\frac{r}{1-u'(r)^2}\, dr=2\pi\int_{1}^{p_R}\frac{r(p)r'(p)}{1-p^2}\, dp\\
&=-2\pi\int_{1}^{p_R}c_1^2\frac{(1-p^2)^2(1+3p^2)}{ p^3}\, dp\\
&=-\frac{\pi  \left(3 p_R^6-10 p_R^4+9 p_R^2+4 p_R^2 \log (p_R)-2\right)}{2 \left(1-p_R^2\right)^4}.
\end{split}
\end{equation}

Without loss of generality, and after a rescaling (Prop. \ref{pr-22}), let $R=1$. We will compute $E[u]$ for different values of $M$. Given $M<R$, the scheme begins by finding the value $p_R$ such that  $r(p_R)=1$. This is obtained from Eq. \eqref{c2}. Once we have calculated $p_R$, we compute the resistance $E[u]$ in \eqref{ee}. We also evaluate the resistance of the cone
 $v(r)=\frac{M}{r}$. This cone is spacelike and also satisfies $v(0) =u(0)=0$ and $v(R) =M$. Its resistance is $\pi \frac{R^2}{r^2-M^2}$.  Fixing $R>0$, Table \ref{table1}  shows some values of the resistance in terms of the domain $B_R$ in comparison with the cone under the same boundary conditions.

\begin{table}[h!]
\centering
\begin{tabular}{|c|c|c|c|c|c|c|c|c|c|}
\hline
$M $& 0.1 & 0.2 & 0.3 & 0.4 & 0.5 & 0.6 & 0.7 & 0.8 & 0.9 \\ \hline
$p_R$ & 0.025 & 0.067 & 0.126 & 0.204 & 0.304 & 0.423 & 0.558 & 0.702 & 0.850 \\ \hline
$E[u]$ & 3.155 & 3.212 & 3.334 & 3.552 & 3.916 & 4.525 & 5.613 & 7.874 & 14.798 \\ \hline
$E[v]$& 3.173&3.272 & 3.452& 3.740 & 4.188&4.908 & 6.159&8.726 & 16.534 \\ \hline
\end{tabular}\vspace{0.2cm}
\caption{Computation of the resistance of a radial solution of Thm. \ref{t51}. A comparison with the spacelike cone is also shown for different values of $M$.   Here $R=1$. }\label{table1}
\end{table}
 
\section{The single shock condition}\label{s6}

To compute the total resistance of the body $\mathcal{B}$, we have imposed the ``Single Shock Property" (SSP). In the Lorentzian context, this condition ensures that, after a particle collides with the surface $S$, the particle does not strike the surface a second time. In the classical problem, under this assumption, the model of energy $E_1$ in \eqref{e1} can be considered accurate to the physical reality. In \cite{bfk}, the authors give a sufficient condition to ensure the SSP, which it is known as the Single Shock Condition (SSC). Assuming the SSC, existence of global minimizers among radial solutions were proved in \cite{cl,cl2}. Considering the class of functions satisfying SSC, minimizers of $E_1$ do not exist \cite{pa0,pa1}.  See also   \cite{ak,lw}.

In this section, we define and analyze the SSC in the Lorentzian context. We come back to Fig. \ref{fig1}. Consider the vertical plane determined by the positive direction $\vec{e}_x$ and $\vec{e}_3$.  We follow the same steps as in \cite{bfk}.     Recall that the final velocity of the particle is $\vec{v}_f=-\sinh\theta T-\cosh\theta N$. In terms of the gradient $Du$,  
\begin{equation*}
\begin{split}
\vec{v}_f&=\vec{e}_3+2\langle \vec{e}_3,N\rangle N=\vec{e}_3-2\cosh\theta N\\
&=\frac{1}{1-|Du|^2}\left(-2 |Du |e_x-(1+|Du|^2) \vec{e}_3\right).
\end{split}
\end{equation*}
Observe that 
$\vec{e}_x=\frac{Du}{|Du|}$. Then  the expression for $\vec{v}_f$ is 
$$\vec{v}_f= \vec{e}_3-\frac{2}{1-|Du|^2}(Du,1)=\frac{1}{1-|Du|^2}\left(-2Du\, \vec{e}_x-(1+|Du|^2)\vec{e}_3\right).$$
If $Du=0$, then $\vec{v}_f=-\vec{e}_3$ and the particle does not hit $\mathcal{B}$ again because $S$ is a graph on the $xy$-plane.  Suppose $Du\not=0$. The slope $s_p$ of the final velocity $\vec{v}_f$ is, up to a sign,  
$$s_p=\frac{1+|Du|^2}{2|Du|}.$$
In   the $xz$-plane we identify  $u(x,y)$ with $u(x)$ and $Du(x,y)$ with $u'(x)$.  If $u'(x)>0$ (resp. $u'(x)<0$), then we say that the SSC  is fulfilled  if 
\begin{equation}\label{sic1}
u(x)+u'(x)(\bar{x}-x)\leq u(\bar{x})\quad\mbox{for all}\quad \bar{x}\leq x,\quad \mbox{(resp. $\bar{x}>x$)}.
\end{equation}
 By taking $\bar{x}=x-tu'(x)$,  then \eqref{sic1} becomes  
$$u(x)-u(x-t u'(x))\leq \frac{t}{2}(1+u'(x)^2),\qquad \mbox{for all $t>0$},$$
or equivalently,
\begin{equation}\label{sic2}
\frac{u(x)-u(x-t u'(x))}{t}\leq \frac{ 1+u'(x)^2)}{2}\quad   (SSC),
\end{equation}
for all $t>0$. In case that $u'(x)<0$, the same inequality \eqref{sic2} holds. Coming back to the function $u(x,y)$, we have the following definition.

\begin{definition}  Let $u\in C^1(\Omega)$ be a function such that $|Du|<1$. The function $u$ is said to satisfy  the SSC if  
\begin{equation}\label{sic}
  u(x,y)-u((x,y)-tDu(x,y))\leq \frac{t}{2}(1+|Du(x,y)|^2) 
\end{equation}
for all $t>0$.
\end{definition}

 The SSC  is a global property in the sense that it must be fulfilled for all $t>0$. If $t\to 0$ in \eqref{sic}, it would be expectable to get extra information about the derivative of $u$ at each point. However, it is not.  
Letting $t\to 0$ in \eqref{sic2},   we deduce $ u'(x)^2\leq \frac{ 1+u'(x)^2)}{2}$, which  is trivially fulfilled because of the spacelike condition $u'(x)^2<1$.    

The fundamental difference from the Euclidean case is that any spacelike function $u$ satisfies the SSC regardless the concavity or convexity of the function.

\begin{theorem} \label{t61}
If $u\in C^1(\Omega)$ is a function such that $|Du|^2<1$, then $u$ satisfies the SSC.
\end{theorem} 

\begin{proof} It suffices to consider the case $|Du|\not=0$. Consider the plane  determined by $\vec{e}_x$ and $\vec{e}_3$. Then we need to check \eqref{sic2} for all points of $\Omega$. Without loss of generality, suppose $u'(x)>0$. Let $t>0$. The Mean Value Theorem implies that there exists $\xi$ in the segment joining $x$ and $x-tu'(x)$  such that 
$$u(x)=u(x-t u'(x))+tu'(x) u'(\xi),$$
for some $ t_0\in [0,t]$. 
Then the left hand-side of \eqref{sic2} is
$$\frac{u(x)-u(x-t u'(x))}{t}=u'(\xi)u'(x).$$
Thus we have to check 
$$u'(\xi)u'(x)\leq \frac{1+u'(x)^2}{2},$$
that is,  $1+u'(x)^2-2u'(\xi)u'(x)\geq 0$. But this inequality is fulfilled  because 
$$1+u'(x)^2-2u'(\xi)u'(x)=(u'(x)-u'(\xi))^2+1-u'(\xi)^2> 0,$$
and $u'(\xi)^2<1$.
\end{proof}


  \section*{Acknowledgements}
Rafael L\'opez  has been partially supported by MINECO/MICINN/FEDER grant no. PID2023-150727NB-I00,  and by the ``Mar\'{\i}a de Maeztu'' Excellence Unit IMAG, reference CEX2020-001105- M, funded by MCINN/AEI/10.13039/ 501100011033/ CEX2020-001105-M.

\section*{Declarations}
{\bf Funding and/or Conflicts of interests/Competing interests}:  Data sharing does not apply to this article as no datasets were generated or analyzed
during the current study. This work has not received any financial support. In addition, the author declares
that he has no conflict/competing of interests.

\end{document}